\documentclass[12pt]{amsart}
\usepackage{amsmath,amssymb}
\def\P{\mathcal{P}}

\DeclareMathAlphabet{\curly}{U}{rsfs}{m}{n}

\newtheorem{theorem}{Theorem}

\newtheorem{lemma}{Lemma}

\newcommand{\RR}{{\mathbb R}}

\newcommand{\NN}{{\mathbb N}}

\newcommand{\g}{\ensuremath{\gamma}}
\newcommand{\del}{\ensuremath{\delta}}
\newcommand{\lam}{\ensuremath{\lambda}}
\newcommand{\eps}{\ensuremath{\varepsilon}}

\newcommand{\fl}[1]{{\ensuremath{\left\lfloor {#1} \right\rfloor}}}

\newcommand{\pfrac}[2]{{\left(\frac{#1}{#2}\right)}}

\newcommand{\ds}{\displaystyle}

\newcommand{\be}{\begin{equation}}
\newcommand{\ee}{\end{equation}}
\newcommand{\benn}{\begin{equation*}}   
\newcommand{\eenn}{\end{equation*}}

\renewcommand{\b}{\ensuremath{\beta}}
\renewcommand{\a}{\ensuremath{\alpha}}

\renewcommand{\(}{\left(}
\renewcommand{\)}{\right)}

\numberwithin{equation}{section}

\title[Diophantine approximation with arithmetic functions]
{Diophantine approximation with arithmetic functions, I}
\author{Emre Alkan, Kevin Ford and Alexandru Zaharescu}
\date{\today}
\thanks{Second author supported in part by National Science Foundation
  Grant DMS-0555367.  Third author supported in part by National
  Science Foundation Grant DMS-0456615.}

\begin{document}

\begin{abstract} We prove a strong simultaneous Diophantine
approximation theorem for values of additive and multiplicative
  functions provided that
the functions have certain regularity on the primes.
\end{abstract}

\maketitle

\section{Introduction}

There is a rich literature on problems of approximating real numbers
by rational numbers with multiplicative restrictions on the
denominator of the rational number, e.g. \cite{AHZ}, \cite{Ba},
\cite{Fr} and the references therein.  We are concerned here with
approximating real numbers by values of additive and multiplicative
functions. One of the classical results in this area is the 1928
theorem of Schoenberg \cite{Scho}, which states that $\phi(n)/n$ has
a continuous distribution function, that is,
$$
F(z)=\lim_{x\to \infty} \frac{1}{x} |\{ n\le x :\phi(n)/n \le z\}|
$$
exists for every real $z$, and $F(z)$ is continuous.
Here $\phi$ is Euler's totient function.
In particular, $\phi(n)/n$ is dense in $[0,1]$, or equivalently,
the additive function $\log \phi(n)/n$ is dense in $(-\infty,0]$
(in general, $f(n)$ is additive if and only if $e^{f(n)}$ is multiplicative).
Erd\H os and Wintner \cite{EW} later determined precisely which real
additive functions have continuous distribution functions.
These include $\log \phi(n)/n$ and its close cousin $\log \sigma(n)/n$,
where $\sigma(n)$ is the sum of the divisors of $n$.
A stronger approximation theorem was proved by Wolke \cite{W}.
Let $\Gamma$ denote the infimum of numbers $\gamma$ so that for large
$x$, there is a prime in  $(x-x^\gamma,x]$.
Wolke proved that for any real $\b \ge 1$ and
any $c<1-\Gamma$, there are infinitely many integers $n$ with
$|\frac{\sigma(n)}{n}-\beta| < n^{-c}$.
It is conjectured that $\Gamma=0$, however we only know that
$\Gamma \le 0.525$ \cite{BHP}.

In the 1950's, several papers appeared concerning the distribution
of values of $\phi(n)$ and the sum of divisors function $\sigma(n)$
at consecutive
integers.  A major unsolved problem is whether, for fixed $k\ne
0$, the equations
$\phi(x+k)=\phi(x)$ or $\sigma(x+k)=\sigma(x)$ have infinitely
many solutions.  For the latest work on the
problem for $\phi$, see \cite{GHP}.   Schinzel \cite{Sc}
proved that for any $h\ge 1$, $\eps>0$ and positive real numbers
$\a_1,\cdots,\a_h$, the system of simultaneous inequalities
\be\label{Schinzel}
\left| \frac{\phi(n+i)}{\phi(n+i-1)} - \a_i \right| < \eps \qquad (1\le i\le h)
\ee
has infinitely many solutions.  Six years later Schinzel teamed with Erd\H os
\cite{ES} to show that \eqref{Schinzel} holds for a positive proportion of
integers $n$, and to generalize this result to a wide class of
additive and multiplicative functions.


It is interesting to ask how fast $\eps=\eps(n)$ can tend to zero as a
function of $n$ in \eqref{Schinzel}.  In particular,
Erd\H os \cite{E} posed the problem to show that
for some $c<1$, the inequalities
$$
|\phi(n+1)-\phi(n)| < n^{c} \qquad \text{ and } \qquad
|\sigma(n+1)-\sigma(n)| < n^{c}
$$
each have infinitely many solutions.

Our first results solve Erd\H os' problem and generalize
the aforementioned
theorems of Schinzel and Wolke.  In particular, we may replace
$\eps$ on the right side of \eqref{Schinzel} with $n^{-c_h}$ for some
positive $c_h$.
As in \cite{ES}, we state our results for a wide class of additive functions.
For any $\delta>0$ and $\lambda>0$ we denote by $\mathcal F_{\delta,\lambda}$
the set of additive functions $f: \NN\rightarrow \RR$ with the
following properties:

(a) We have
\benn
\sum_{\substack{p\,\,{\rm prime}\\f(p)>0}}f(p) = \infty.
\eenn

(b) There exists a constant $C(f)>0$ depending on $f$ such that
\benn
\left|f(p^v)\right|\leq \frac{C(f)}{p^{\delta}}\,,
\eenn
for any prime number $p$ and $v\ge 1$.

(c) There exists $t_0(f)>0$ depending on $f$ such that for any
$0<t\leq t_0(f)$ there is a prime number $p$ satisfying
\begin{equation*}
t-t^{1+\lambda}\leq f(p)\leq t.
\end{equation*}

We remark that (a) and (b) imply that $\delta \le 1$.
Also, (b) and (c) together imply that $\delta \lambda < 1$.
To see this, let $\eps>0$  and $S=|\{ p : f(p)>\eps\}|$.
By (b),
$$
S \le | \{ p : C(f)/p^\delta > \eps \}| = o(\eps^{-1/\del}) \quad
(\eps\to 0^+).
$$
On the other hand (cf. \eqref{vjbound} below),
the interval $[\eps,t_0(f)]$ contains $\gg \eps^{-\lam}$ disjoint intervals
of the form $(t-t^{1+\lam},t]$.  Hence $S \gg \eps^{-\lam}$.
Finally, if (c) holds with $\lambda>1$ then (a) follows.

\begin{theorem}\label{Theorem1}
Fix an integer $k\geq1$ and real numbers $0<\delta\le 1$, $0<\lambda<1/\del$.
Let $f_1,\dots,f_k$ be functions in $\mathcal F_{\delta,\lambda}$,
and let $A=1$ if all $f_i$ are identical, and $A=2$ otherwise.
Suppose $a_1,\ldots,a_k$ are positive integers, $b_1,\ldots,b_k$
are nonzero integers and
\be\label{aibi}
a_i b_j \ne a_j b_i \quad (1\le i < j \le k).
\ee
If $\a_i>f_i(b_i)$ for
$1\leq i\leq k$ and
\be\label{e1}
0<c< \begin{cases} \del\lam & {\rm if}\;\; k=1 \text{ and } a_1\in\{1,2\} \\
\frac{\delta\lambda}{A k+\lambda\b_{k}} & {\rm otherwise}\,, \end{cases}
\ee
then there are infinitely many positive integers $m$ satisfying
\be\label{e2}
\left| f_i(a_i m + b_i)-\a_i\right|<\frac1{m^c} \qquad (1\le i\le k).
\ee
Here $\b_k$ is an admissible value of
a ``sieve limit for a sieve of dimension $k$'', which will be described below
(Theorem DHR).  In particular, $\b_1=2$, $\b_2<4.2665$ and $\b_k=O(k)$.
\end{theorem}

\begin{theorem}\label{Theorem2}
Fix an integer $k\geq1$ and real numbers $0<\delta\le 1$, $0<\lambda<1/\del$.
Let $f_0,\dots,f_{k}$ be functions in
$\mathcal F_{\delta,\lambda}$, with $A=1$ if $f_1,\ldots,f_k$ are
identical and $A=2$ otherwise.
Suppose $a_0,a_1,\ldots,a_k$ are positive integers, $b_1,\ldots,b_k$
are integers and \eqref{aibi} is satisfied.  If
$\zeta_1,\dots,\zeta_k$ are arbitrary real numbers, and
\benn
0<c< \begin{cases} \frac{\del\lam}{1+4\lam} & {\rm if}\;\; k=1 \\
\frac{\delta\lambda}{Ak+\lambda \beta_{k+1}} & {\rm if}\;\; k\ge 2, \end{cases}
\eenn
then there are infinitely many positive integers $m$ satisfying
\benn
\left| f_{i}(a_im+b_i)-f_{i-1}(a_{i-1}m+b_{i-1})-\zeta_i\right|
<\frac1{m^c} \qquad (1\le i\le k).
\eenn
\end{theorem}

{\bf Remarks.}  Theorem \ref{Theorem1} implies immediately the
conclusion of Theorem \ref{Theorem2} in the range
$$
0<c<\frac{\delta\lambda}{A(k+1)+\lambda \beta_{k+1}}\,
$$
by choosing $\a_0,\ldots,a_k$ large and satisfying
$\a_i-\a_{i-1}=\zeta_i$ for each $i$.
Larger values of $c$ are possible by making
a more judicious choice of $\a_0$.

If we assume the Elliott-Halberstam conjecture on the distribution of
primes in arithmetic progressions, then the conclusion of Theorem 2
for $k=1$ holds for
\benn
0<c<\frac{\delta\lambda}{1+2\lam}.
\eenn

\bigskip

{\bf Corollaries.}
We may apply Theorems \ref{Theorem1} and \ref{Theorem2} to the functions
$f(n)=\log (n/\phi(n))$ and $f(n)=\log(\sigma(n)/n)$.  Each of these satisfies
$f(p) = \frac{1}{p} + O(\frac{1}{p^2})$.
It follows that $f \in \mathcal F_{1,\lam}$ for any $\lam<1-\Gamma$.
Applying Theorem \ref{Theorem1} with $k=1$, $a_1=1$ and
$f_1(n)=\log(\sigma(n)/n)$ recovers Wolke's result.
Applying Theorem 2 with $f_i(n)=\log(n/\phi(n))$ shows that one may
take $\eps=n^{-c_h}$ in \eqref{Schinzel} provided that
$$
c_h < \frac{1-\Gamma}{h+(1-\Gamma)\b_{h+1}}.
$$

We also have an answer to Erd\H os' question, by
applying Theorem \ref{Theorem2} with $k=1$, $\zeta_1=0$,
$a_0=a_1=b_1=1$ and $b_0=0$.
 For any $c<\frac{1-\Gamma}{5-4\Gamma}$, the inequalities
$$
|\phi(n+1)-\phi(n)| < n^{1-c}, \qquad |\sigma(n+1)-\sigma(n)| <
n^{1-c}
$$
each have an infinite number of solutions.  In addition, for
any nonzero $a$,  $c<\frac{1-\Gamma}{5-4\Gamma}$
and for any real $\zeta$, the inequality
$$
|n/\phi(n+a)-\sigma(n)/n-\zeta| < n^{-c}
$$
has infinitely many solutions.

The methods used to prove Theorems 1 and 2 also yield similar results
for the simultaneous approximation of $f_i(g_i(n))$ where $g_i(n)$
are polynomials, provided that (a) and (c) above are suitably
strengthened.  Rather than aim for fullest generality, we illustrate
what is possible with two special cases.

\begin{theorem}\label{poly}  Suppose $h(n)=\phi(n)$ or
$h(n)=\sigma(n)$.   Let
$\Gamma'$ be the infimum of numbers $g$ so that if $x$ is large,
there is a prime $p\equiv 1\pmod{4}$ with $x-x^g < p \le x$.
For any real $\zeta$ and any
$$
c<\frac{1-\Gamma'}{1+(1-\Gamma')\beta_2},
$$
the inequality
$$
|h(n^2+1)-h(n^2+2)-\zeta| < n^{2-c}
$$
has infinitely many solutions.
\end{theorem}

A sketch of the proof of Theorem  \ref{poly} will appear in section
\ref{sec:poly}, together with a discussion of how to deal with more
general $f_i$ and $g_i$. Likely the methods in \cite{BHP} can be
used to prove $\Gamma' \le 0.525$, but the best result available in
the literature is $\Gamma' \le 0.53$ (by considering the polynomial
$Q(x,y)=x^2+y^2$ in \cite{HKL}).

In \cite{LS}, a similar Diophantine
approximation problem is considered
for consecutive values of the kernel function
$k(n)=\prod_{p|n} p$.  Our methods do not apply, since $f(p)=0$ for
$f(n)=\log (n/k(n))$.  Luca and Shparlinski \cite{LS} show that for any vector
$(\a_1,\ldots,\a_k)$ of positive real numbers, there are infinitely
many $n$ for which
$$
\left| \frac{k(n+i-1)}{k(n+i)} - \a_i \right| < \frac{1}{n^{1/41k^3}}
\qquad (1\le i\le k-1).
$$

In a sequel paper, we will consider Diophantine approximation problems for
coefficients of modular forms.  A example of one of our results is that
for any real $\b$, there is a constant $C_\b$ so that for infinitely
many $n$,
$$
\left| \frac{\tau(n)}{n^{11/2}} - \b \right| \le \frac{C_\b}{\log n},
$$
where $\tau(n)$ is Ramanujan's function, the $n$th coefficient of
$q \prod_{m=1}^\infty (1-q^m)^{24}$.

%
\section{Preliminaries for Theorems \ref{Theorem1}, \ref{Theorem2}
and \ref{poly}}
%

\begin{lemma}\label{Lemma1}
Let $0<\del \le 1$, $0<\lam<1/\del$, $f_1,\cdots,f_k \in
\mathcal F_{\del,\lam}$, $0<\xi<\lam/A$, and $K\ge 1$.  For sufficiently small
positive $v_0$, there are
disjoint sets $\P_1,\ldots,\P_k$ of primes greater than $K$
with the following properties.
(i) Let $v_{j+1}=v_j-v_j^{1+\xi}$ for $j\ge 0$.   For each $j\ge 0$
and $1\le i\le k$, $\P_i$ contains exactly one prime
$p$ with $f_i(p)\in (v_{j+1},v_j]$.
(ii) Let $\P_0$ be the set of primes larger than $K$ which do not lie
  in any set $\P_i$.  Then
\be\label{P0sum}
\sum_{\substack{p\in \P_0 \\ f_i(p)>0}} f_i(p) = \infty \quad (1\le
i\le k).
\ee
\end{lemma}

\begin{proof}
It is straightforward to show that for any $v_0$,
\be\label{vjbound}
v_j \sim \xi^{-1/\xi} j^{-1/\xi} \qquad (j\to \infty).
\ee
One method of proof is to compare $v_j$ to $y(j)$, where $y$
satisfies the differential equation $y' = -y^{1+\xi}$.
We will take $v_0$ satisfying
\begin{align}
v_0 &< \min_{1\le i\le k} t_0(f_i),  \label{v0a}\\
v_0 &< \min_{1\le i\le k} \min\{ f_i(p) : p\le K\text{ and } f_i(p)>0
\}. \label{v0b}
\end{align}
If $f_1,\ldots,f_k$ are identical, we also assume that
\be\label{v0c1}
v_0^{\xi-\lam} \ge 2k.
\ee
Since $\xi<\lam$, $v_0$ satisfies \eqref{v0a},
\eqref{v0b}, and \eqref{v0c1} if $v_0$ is small enough.
If $f_1,\ldots,f_k$ are not identical, then $\xi < \frac12 \lam$.
By \eqref{vjbound}, if $v_0$ is small enough then
\be\label{v0c2}
v_j^{\xi-\lam} \ge 2k^2(j+1) \qquad (j\ge 0).
\ee

Next, we construct the sets $\P_i$.  First assume $f_1,\cdots,f_k$ are
identical.  By \eqref{v0c1}, for each $j\ge 0$ the interval
$(v_{j+1},v_j]$ contains at least $2k$ disjoint intervals of the form
$(v-v^{1+\lam},v]$.  By (c) and \eqref{v0a}, each such interval
 contains a value of $f_1(p)$ for some prime
$p$.  Label these $2k$ primes $p_{i,j}, p'_{i,j}$ for $1\le i\le k$.

Assume that $f_1,\cdots,f_k$ are not all identical.
Fix $j\ge 0$ and assume that we have chosen distinct primes
$p_{i,h}, p'_{i,h}$ such that
$f_i(p_{i,h}), f_i(p'_{i,h})\in (v_{h+1},v_h]$
for $1\le i\le k$, $0\le h\le j-1$.
By \eqref{v0c2}, $(v_{j+1},v_j]$ contains at least $2k^2(j+1)$
intervals of the form $(v-v^{1+\lam},v]$.
At most $2k(k-1)j$ of these intervals contain a number of the
form $f_{i'}(p_{i,h})$ or $f_{i'}(p'_{i,h})$ for $1\le i\le k$,
$1\le i'\le k$, $i\ne i'$, $0\le h\le j-1$.  Let $T$ denote the set of
remaining intervals, so that $|T| > 2k^2$.
Take two intervals $I_1,I'_1\in T$.  By \eqref{v0a} and (c), there are
primes $p_{1,j}, p'_{1,j}$ with $f_1(p_{1,j})\in I_1$
and $f_1(p'_{1,j})\in I'_1$.  Take 4 intervals in $T\backslash \{I_1,I'_1\}$.
By \eqref{v0a} and (c), there are two of these intervals $I_2$ and $I'_2$ and
primes $p_{2,j}, p'_{2,j}$ different from $p_{1,j}$ and $p'_{1,j}$
so that $f_2(p_{2,j}) \in I_2$ and $f_2(p'_{2,j})\in I'_2$.
Continuing this process, since
$|T| \ge 2 + 4 + \cdots + 2k$, we can find $2k$ distinct intervals
$I_1, I'_1, \ldots, I_k, I'_k \in T$ and $2k$ distinct primes
$p_{1,j}, p'_{1,j}, \ldots, p_{k,j}, p'_{k,j}$, different from
the previously chosen primes $p_{i,h},p'_{i,h}$ ($1\le i\le k, h<j$) with
$f_i(p_{i,k})\in I_i$ and $f_i(p'_{i,k})\in I'_i$ for $1\le i\le k$.

In either case, for $1\le i\le k$ let
$$
\P_i = \{ p_{i,0}, p_{i,1}, \ldots \}.
$$
By \eqref{v0b}, all primes $p_{i,j}, p'_{i,j}$ are larger than $k$.
Hence, the sets $\P_1,\cdots,\P_k$ satisfy condition (i) of the lemma.
If $\sum_{j\ge 0} v_j$ converges (that is, $\xi<1$), then for each $i$,
$$
\sum_{p\in \P_i} f_i(p) < \infty,
$$
and hence by (a), \eqref{P0sum} holds.
Next assume $\sum_{j\ge 0} v_j$ diverges.  For each $i$ and every
$j\ge 0$, $p'_{i,j} \in \P_0$ and $f_i(p'_{i,j})\in
(v_{j+1},v_j]$. Hence \eqref{P0sum} holds in this case as well.
\end{proof}

%
%

\begin{lemma}\label{Lemma2}
Let $0<\del \le 1$, $0<\lam<1/\del$, $f_1,\cdots,f_k \in \mathcal
F_{\del,\lam}$, and $0<\xi<\lam/A$. Also assume that $K\ge 1$ and
that $\g_1,\ldots,\g_k$ are positive real numbers. For sufficiently
small $v_0$ and for \be\label{eta} 0 < \eta < \min\left(v_0,
6^{-1/\xi}, \g_1, \ldots, \g_k\right), \ee there are sequences
$\{n_{i,j}\}, 1\leq i\leq k$, $j=0,1,2,\dots$, such that
\begin{enumerate}
\item $n_{i,j} | n_{i,j+1}$ for each $1\leq i\leq k$ and $j\ge 0$;
\item For each $j$, the numbers $n_{1,j}$, $n_{2,j},\dots, n_{k,j}$
are pairwise relatively prime and divisible by no prime $\le K$;
\item $|f_i(n_{i,j}) - \g_i| \le 3^j \eta^{(1+\xi)^j}$ for $1\le
i\le k$, $j\ge 0$;
\item we have
\be\label{nij}
n_{i,j}\leq \frac{(2C(f_i))^{\frac{j}{\delta}} n_{i,0}}{(\eta/2)
^{\frac{(1+\xi)^j}{\del\xi}}} \qquad
(1\le i\le k, j\ge 0).
\ee
\end{enumerate}
\end{lemma}

Assume \eqref{eta} is satisfied, and let
$\P_0, \P_1, \ldots, \P_k, v_1,\ldots$  be as in Lemma \ref{Lemma1}.
By \eqref{P0sum} and (b),
there are prime numbers $q_1, q_2, \dots, q_r$ in $\P_0$
such that
\begin{equation*}
\gamma_1-\eta<\sum_{s=1}^r f_1(q_s)<\gamma_1-\frac{\eta}2\,.
\end{equation*}
We take
\begin{equation*}
n_{1,0}=\prod_{s=1}^rq_s
\end{equation*}
so that $\gamma_1-\eta< f_1(n_{1,0})<\gamma_1-\frac{\eta}2\,.$
In this way we may successively construct $n_{i,0}$
for $2\leq i\leq k$.
Assume that we have already constructed
$n_{1,0}, n_{2,0}, \dots, n_{i-1,0}$ with prime divisors in
$\mathcal{P}_0$. If $\mathcal{B}_i$ is the set of all prime divisors of
$n_{1,0}, n_{2,0}, \dots, n_{i-1,0}$, then by \eqref{P0sum},
\begin{equation*}
\sum_{p\in\mathcal{P}_0\setminus\mathcal{B}_i}f_i(p)=\infty.
\end{equation*}
Therefore, by (b) we may choose $n_{i,0}$ with all prime divisors in
$\mathcal{P}_0\setminus\mathcal{B}_i$ and with
 $\gamma_i-\eta< f_i(n_{i,0})<\gamma_i-\frac{\eta}2$.

Next we construct $n_{1,j}$ for $j\ge 1$.
Put $\tau_{1,0}=\gamma_1-f_1(n_{1,0})\in (\frac{\eta}2,\eta)$ and
recursively define $n_{1,j+1}=n_{1,j}p_{1,j+1}$ for $j\geq0$
where $p_{1,1}, p_{1,2}, \dots $ are prime numbers to be chosen from
$\P_1$. Clearly we have
$f_1(n_{1,j+1})= f_1(n_{1,j}) + f_1(p_{1,j+1})$ and in particular
$f_1(n_{1,1})= f_1(n_{1,0}) + f_1(p_{1,1})$. Consider the interval
\begin{equation*}
I_{1,0}=\Big(\tau_{1,0}-\tau_{1,0}^{1+\xi}-2(\tau_{1,0}-
\tau_{1,0}^{1+\xi})^{1+\xi},\tau_{1,0}-\tau_{1,0}^{1+\xi}\Big].
\end{equation*}
Since
\begin{equation*}
\tau_{1,0}-\tau_{1,0}^{1+\xi}-(\tau_{1,0}-
\tau_{1,0}^{1+\xi})^{1+\xi}-\Big(\tau_{1,0}-\tau_{1,0}^{1+\xi}-
(\tau_{1,0}- \tau_{1,0}^{1+\xi})^{1+\xi}\Big)^{1+\xi}
\end{equation*}
\begin{equation*}
>\tau_{1,0}-\tau_{1,0}^{1+\xi}-2(\tau_{1,0}-\tau_{1,0}^{1+\xi})
^{1+\xi},
\end{equation*}
$I_{1,0}$ contains an interval of form $(v_{j+1},v_j]$ with
$v_{j+1}=v_j-v_j^{1+\xi}$.
Therefore, we can find $p_{1,1}$ in
$\P_1$ satisfying
\begin{equation*}
\tau_{1,0}-3\tau_{1,0}^{1+\xi}< \tau_{1,0}-\tau_{1,0}^{1+\xi}-
2(\tau_{1,0}-\tau_{1,0}^{1+\xi})^{1+\xi}
<f_1(p_{1,1})\leq \tau_{1,0}- \tau_{1,0}^{1+\xi}.
\end{equation*}
Let $\tau_{1,1}=\gamma_1-f_1(n_{1,1})=\gamma_1-f_1(n_{1,0})-
f_1(p_{1,1})=\tau_{1,0}-f_1(p_{1,1})$ so that
\begin{equation*}
\tau_{1,0}^{1+\xi}\leq \tau_{1,1} < 3\tau_{1,0}^{1+\xi}.
\end{equation*}
Inductively we can find prime numbers $p_{1,1}, p_{1,2}, \dots$
in $\P_1$ such that
$\tau_{1,j}=\gamma_1-f_1(n_{1,j})$ and
\begin{equation*}
\tau_{1,j-1}^{1+\xi}\leq \tau_{1,j} < 3\tau_{1,j-1}^{1+\xi}
\end{equation*}
for $j\geq1$.  Since $\tau_{1,0} < \eta < 6^{-1/\xi}$,
the intervals
$[\tau_{1,j}-3\tau_{1,j}^{1+\xi},\tau_{1,j}-\tau_{1,j}^{1+\xi}]$ are
disjoint. Consequently, the prime numbers
$p_{1,1}, p_{1,2}, \dots $ that are chosen at each step from $\P_1$
are distinct. By iterating the inequalities we also have
\begin{equation*}
\tau_{1,0}^{(1+\xi)^j}\leq \tau_{1,j} \leq 3^j\tau_{1,0}^{(1+\xi)^j}
\end{equation*}
for any $j\geq 0$. Moreover,
\begin{equation*}
\tau_{1,j-1}-3\tau_{1,j-1}^{1+\xi} < f_1(p_{1,j}) \leq
\frac{C(f_1)}{p_{1,j}^{\delta}}
\end{equation*}
and it follows that
\begin{equation*}
p_{1,j}\leq \left(\frac{C(f_1)}{\tau_{1,j-1}-3\tau_{1,j-1}^{1+\xi}}
\right)^{\frac1{\delta}}\leq \left(\frac{2C(f_1)}{\tau_{1,j-1}}
\right)^{\frac1{\delta}}
\end{equation*}
for any $j\geq1$. It follows that
for any $j\geq1$,
\begin{equation*}
n_{1,j}\leq\left(\frac{(2C(f_1))^{\frac{j}{\delta}}}{\Big(\prod_{s=0}^{j-1}
\tau_{1,s}\Big)^{\frac1{\delta}}}\right)n_{1,0}\,.
\end{equation*}
Using the fact that
\begin{equation*}
\prod_{s=0}^{j-1}\tau_{1,s}\geq\tau_{1,0}^{\sum_{s=0}^{j-1}(1+\xi)^s}
\geq\tau_{1,0}^{(1+\xi)^j\sum_{s=1}^{\infty}\frac1{(1+\xi)^s}}
\geq\left(\frac{\eta}2\right)^{\frac{(1+\xi)^j}{\xi}},
\end{equation*}
we obtain
\begin{equation}\label{nija}
n_{1,j}\leq\left(\frac{(2C(f_1))^{\frac{j}{\delta}}}{\left(\frac{\eta}2
\right)^{\frac{(1+\xi)^j}{\xi\delta}}}\right)n_{1,0}
\end{equation}
for any $j\geq0$.

We construct $n_{i,j}$ for $2\le i\le k$, $j\ge 1$ in a similar manner.
More precisely, $n_{i,j}=n_{i,j-1}p_{i,j}$ and the $p_{i,j}$'s are distinct
primes in $\mathcal{P}_i$ for each $j\geq1$.
Conditions (i) and (ii) are immediate.  Condition (iii) follows from
$$
0 < \g_i - f_i(n_{i,j}) \le 3^j \tau_{i,0}^{(1+\xi)^j} \le 3^j
\eta^{(1+\xi)^j}
$$
and \eqref{nij} follows from \eqref{nija}.

%
\section{Proof of Theorems \ref{Theorem1} and \ref{Theorem2}}
%

Suppose $0 < \xi' < \xi < \lam/A$ and $v_0$ is sufficiently small.
Put $L=(2k!b_1\cdots b_k)^2$, let $K$ be the largest prime factor of
$L$ and define $\g_j=\a_j-f_j(b_j)$ for $1\le j\le k$.
$\eta$ satisfies \eqref{eta} and also
\be\label{etaxi}
(\eta/2)^{\xi'} > \eta^\xi.
\ee
Let $n_{i,j}$ $(1\le i\le k, j\ge 0)$ be the sequences of
integers guaranteed by Lemma \ref{Lemma2}.

If $k=1$ and $a_1\in\{1,2\}$ in Theorem \ref{Theorem1},
then $A=1$, each $n_{i,j}$ is odd and for large enough $j$,
\begin{align*}
|f_1(b_1(n_{1,j}-1)+b_1)-\a_1| &= |f_1(b_1 n_{1,j}) - \a_1| \\
&\le  \eta^{(1+\xi)^j} 3^j \\
&= 3^j \pfrac{\eta}{2}^{(1+\xi)^j\xi'/\xi} \left[ \eta
  \pfrac{2}{\eta}^{\xi'/\xi} \right]^{(1+\xi)^j} \\
&\le \pfrac{\eta}{2}^{(1+\xi)^j\xi'/\xi} (2C)^{-j\xi'} n_{1,0}^{-\xi'\del} \\
&\le n_{i,j}^{-\xi' \del}
\end{align*}
by \eqref{etaxi}.
Theorem \ref{Theorem1} follows by taking $\xi,\xi'$ so that $c<\xi'\del$.
The above argument fails when $a_1>2$ because we cannot guarantee that
infinitely many numbers $n_{1,j}$ are congruent to 1 modulo $a_1$,
although this can be done in some cases, e.g. $f_1(n)=\log (n/\phi(n))$.

When $k\ge 2$ or when $k=1$ and $a_1>2$ in Theorem \ref{Theorem1},
take $j$ large and consider the system of congruences
\begin{align*}
m &\equiv 0 \pmod {L} \\
a_1m+b_1 &\equiv 0 \pmod {n_{1,j}} \\
a_2m+b_2 &\equiv 0 \pmod {n_{2,j}} \\
& \quad \dots \\
a_km+b_k &\equiv 0 \pmod {n_{k,j}}.
\end{align*}

By the Chinese remainder theorem, this system is equivalent to a single
congruence $m\equiv h_j \pmod{N_j}$, where
\begin{equation*}
N_j=L\prod_{i=1}^k n_{i,j}
\end{equation*}
and $0\le h_j < N_j$.
\noindent We show that there is a solution $m$ to the above
system of congruences such that all the prime factors of
\begin{equation*}
M=\prod_{i=1}^k\frac{a_im+b_i}{b_i\, n_{i,j}}
\end{equation*}
are large.  This is accomplished with a lower bound sieve.
We use the following theorem of Diamond, Halberstam and Richert
(\cite{DHR1}, \cite{DHR2}).

\medskip


{\bf Theorem DHR.} {\it Let $\mathcal{A}$ be a finite set of positive
integers, $\mathcal{P}$ a set of primes and let
$S(\mathcal{A},\mathcal{P})$ be the number of integers in
$\mathcal{A}$ not divisible by any prime in $\mathcal{P}$.
Let $P(z)$ be the product of the primes in $\mathcal{P}$ which are $\le z$.
For real $\kappa\ge 1$, there is a continuous, increasing function
$f_\kappa$ so that if $X\ge y\ge z\ge 2$ and $\omega$ is a multiplicative function
satisfying $0\le \omega(p)<p$ for $p\in \P$, $\omega(p)=0$ for
$p\not\in \mathcal{P}$ and
\be\label{dimension}
\prod_{v\le p<w} \( 1 - \frac{\omega(p)}{p} \)^{-1} \le \pfrac{\log
w}{\log v}^\kappa \( 1 + \frac{A}{\log v} \) \quad (2\le v\le w),
\ee
then
\begin{multline*}
S(\mathcal{A},\mathcal{P})\geq X\prod_{\substack{p\in\mathcal{P}\\
p\leq z}}\left(1-\frac{\omega(p)}{p}\right)\left( f_\kappa\pfrac{\log
y}{\log z} + O_{\kappa,A}\pfrac{\log\log y}{(\log y)^{1/(2\kappa+2)}}
\) \\
- \sum_{\substack{d|P(z) \\ d<y}} (1+4^{\nu(d)}) |r_d|,
\end{multline*}
where $\nu(d)$ is the number of prime factors of $d$ and
\begin{equation*}
r_d = \#\{ n\in\mathcal{A} : d|n\} - \frac{\omega(d)}{d} X.
\end{equation*}
Here the constant implied by the $O-$symbol depends on $\kappa$ and $A$ only.
Moreover, $f_\kappa(u)>0$ for $u>\b_\kappa$, where $\b_\kappa$ is a certain
constant (see e.g. Appendix III of \cite{DHR1}).  In particular,
$\b_1=2$, $\b_2 < 4.2665$ and $\b_k = O(k)$.
}

\medskip

To apply the theorem, we take
$\kappa=k$, $\mathcal{P}$ the set of all primes $\leq z$, and
\begin{equation*}
  \mathcal{A}=\{P(s) : 1\le s \le N_j^\mu\}
\end{equation*}
where
$$
P(s)=\prod_{i=1}^k \frac{a_i(sN_j+h_j)+b_i}{b_i\, n_{i,j}}
$$
and $\mu$ is a positive constant.
Take $X=N_j^{\mu}$ and
\begin{equation*}
\omega(d)=\#\{ 0\le s\le d-1 : \; P(s)\equiv 0\pmod{d}\}.
\end{equation*}
Then $\omega(p)=0$ for $p|L$, and by \eqref{aibi},
$\omega(p) \le k$ for other $p$.
Thus, by Mertens' estimates,
\eqref{dimension} holds with $\kappa=k$ and $A$ some
constant depending only on $k$.  Let $\eps>0$ and $y=X^{1-2\eps}$.  Since
$$
4^{\nu(d)}|r_d|\leq 4^{\nu(d)} \omega(d) \le (4k)^{\nu(d)}
\ll_{\eps}d^{\eps},
$$
we find that
$$
\sum_{\substack{d|P(z) \\ d<y}} (1+4^{\nu(d)}) |r_d| \ll X^{1-\eps}.
$$
Take $z=y^{\frac{1-\eps}{\b_k}}=N_j^{c_0}$,
$c_0=\frac{\mu(1-2\eps)(1-\eps)}{\b_k}$.  We find that for large $j$
$$
S(\mathcal{A},\mathcal{P}) \gg_{k,\mu,\eps} \frac{N_j^\mu}{(\log N_j)^{k}},
$$
Thus, there is an integer $m\le N_j^{1+\mu}+h_j$ such
that $m \equiv 0\pmod{L}$, $a_im+b_i \equiv 0\pmod{n_{i,j}}$ for
$1\leq i\leq k$ and all prime factors of
\begin{equation*}
M=\prod_{i=1}^k\frac{a_im+b_i}{b_i\, n_{i,j}}
\end{equation*}
\noindent are $> z$.  There are at most $\fl{\frac{1+\mu}{c_0}+1}$
prime factors of $a_im+b_i$ which are $> z$.  By (b), for $1\le i\le k$ we
have
$$
|f_i(a_im+b_i)-\a_i| \le |f(b_i n_{i,j}) - \a_i| + |f(\tfrac{a_im+b_i}{b_i
n_{i,j}})| \le 3^j \eta^{(1+\xi)^j} + c_1 N_j^{-\del c_0},
$$
where $c_1 = (\frac{1+\mu}{c_0}+1) \max_{1\le i\le k} C(f_i)$.
Moreover, by \eqref{nij} and \eqref{etaxi}, for large $j$ we have
\begin{align*}
N_j &= L \prod_{i=1}^k n_{i,j} \le
L  (2\max_{1\le i\le k} C(f_i))^{kj/\del}
(\eta/2)^{-\frac{k(1+\xi)^j}{\del\xi}}
\prod_{i=1}^k n_{i,0} \\
&\le 3^{-kj/(\del\xi')} \eta^{-\frac{k(1+\xi)^j}{\del\xi'}}.
\end{align*}
We conclude that for large $j$,
\begin{align*}
|f_i(a_im+b_i)-\a_i| &\le N_{j}^{-\del \xi'/k} + c_1 N_j^{-\del c_0} \\
&\ll  m^{-\frac{\del \xi'}{k(1+\mu)}} + m^{-\frac{\del
\mu(1-3\eps)}{\b_k(1+\mu)}}.
\end{align*}
Taking $\mu= \frac{\xi'\b_k}{k(1-3\eps)}$ gives
$$
|f_i(a_im+b_i)-\a_i| \ll  m^{-c_2} \qquad (1\le i\le k),
$$
where
$$
c_2 = \frac{\del \xi'}{k+\xi' \b_k (1-3\eps)^{-1}}.
$$
Theorem \ref{Theorem1} follows by taking $\eps$ sufficiently small and
$\xi'$ sufficiently close to $\lam/A$, so that $c_2>c$.

\bigskip\bigskip

{\it Proof of Theorem \ref{Theorem2}.}  Without loss of generality, we
may assume that $b_i>0$ for all $i$.
Let $L=(2k! a_0b_0\cdots a_kb_k)^2$.
By (a), there is a number $n_0$ with $(n_0,L)=1$ and
$$
f_0(b_0 n_0) > \sum_{i=1}^k |\zeta_i| + \max_{1\le i\le k} |f_i(b_i)|.
$$
Let $\a_0=f_0(n_0 b_0)$, $\a_i=\zeta_i+\a_{i-1}$ and
$\g_i=\a_i-f_i(b_i)$ for $1\le i\le k$..  Then $\g_i>0$ for
$1\le i\le k$.  Let $0<\xi<\xi'<\lam/A$, $v_0$ be sufficiently small
such that
$$
v_0 < \min_{p|n_0, f_0(p)>0} f_0(p),
$$
and suppose $\eta$ satisfies \eqref{eta} and \eqref{etaxi}.
Let $K$ be the largest prime factor of $Ln_0$, and let $n_{i,j}$ be as
in Lemma \ref{Lemma2}.  Consider the system
\begin{align*}
m &\equiv 0 \pmod{L} \\
a_0m+b_0 &\equiv 0 \pmod{n_0} \\
a_1m+b_1 &\equiv 0 \pmod{n_{1,j}} \\
&\cdots \\
a_km+b_k &\equiv 0 \pmod{n_{k,j}}
\end{align*}
which is equivalent to a single congruence $m\equiv h_j \pmod{N_j}$,
where $N_j=L n_0 n_{1,j}\cdots n_{k,j}$ and $0\le h_j < N_j$.
Write $m=h_j+sN_j$.

If $k\ge 2$, we apply Theorem DHR with
$$
\mathcal{A}=\{P(s) : 1\le s \le N_j^\mu\}
$$
where
$$
P(s)=\frac{a_0(sN_j+h_j)+b_0}{b_0 n_0} \prod_{i=1}^k
\frac{a_i(sN_j+h_j)+b_i}{b_i\, n_{i,j}},
$$
and $\mathcal{P}$ is the set of primes $\le z$.  Take $X=N_j^\mu$,
$y=X^{1-2\eps}$ and $z=y^{\frac{1-\eps}{\b_{k+1}}}$.
The remaining argument is nearly identical to that in the proof of Theorem
\ref{Theorem1}.  The only differences are that $N_j$ is a factor $n_0$
 larger than before, $\kappa=k+1$, we take $\mu = \frac{\xi'
   \b_{k+1}}{k(1-3\eps)}$, and
$$
|f_0(a_0m+b_0)-\a_0| = \left| f_0\( \frac{a_0m+b_0}{b_0n_0} \) \right|
\ll z^{-\del}.
$$
We find that
$$
|f_i(a_im+b_i)-\a_i| \ll m^{-c_2} \qquad (0\le i\le k)
$$
where
$$
c_2 = \frac{\del \xi'}{k+\xi' \b_{k+1} (1-3\eps)^{-1}}.
$$
Taking $\eps$ sufficiently small and $\xi'$ sufficiently close to $\lam/A$
completes the proof.

If $k=1$, we set up the sieve procedure differently.  Let
$$
q = \frac{a_1 L n_0}{b_1}, \qquad r = \frac{a_1 h_j + b_1}{b_1 n_{1,j}}.
$$
We will restrict our attention to numbers $m$ so that
$\frac{a_1m+b_1}{b_1 n_{1,j}}=qs+r$ is prime.  Apply
Theorem DHR with
$$
\mathcal{A} = \left\{ \frac{a_0(h_j+Ln_0n_{1,j} s)+b_0}{b_0n_0} :
N_j^{\mu} < s\le 2N_j^\mu, qs+r\text{ is prime} \right\}
$$
and $\mathcal{P}$ is the set of primes $\le z$.  Take
$X=\frac{1}{\phi(q)} ( \text{li}(2qN_j^{\mu}+r) - \text{li}(qN_j^{\mu}+r))$,
$y=X^{1/2-\eps}$ and $z=y^{(1-\eps)/\b_1}=y^{(1-\eps)/2}$
for some small fixed $\eps>0$.  Here
$$
\text{li}(x) = \int_2^x \frac{dt}{\log t}.
$$
Each set $\{w\in\mathcal{A} : d|w \}$ is either empty, and we take
$\omega(d)=0$, or
counts primes in a single progression modulo $qd$
which are between $qN_j^{\mu}+r$ and $2qN_j^{\mu}+r$, in which case
we take $\omega(d)=\phi(q)/\phi(qd)$.
Then \eqref{dimension}
holds with $\kappa=1$ and some absolute constant $A$.
The Bombieri-Vinogradov theorem (e.g. Ch. 28 of \cite{Da}) implies
that
\be\label{rem}
\sum_{d\le y} (1+4^{\omega(d)}) |r_d| \ll \frac{X}{\log^5 X}.
\ee
Therefore, by Theorem DHR, if $j$ is large
 then there is a number $s$, $N_j^{\mu}<s\le 2N_j^{\mu}$
with $qs+r$ prime and
all prime factors of $\frac{a_0(h_j+Ln_0n_{1,j} s)+b_0}{b_0n_0}$
are $>z$.  For $m=h_j+Ln_0 n_{1,j} s$, we therefore have by (b),
\begin{align*}
|f_0(a_0m+b_0)-\a_0| &= \left| f_0\(
\frac{a_0(m_0+Ln_0n_{1,j}s)+b_0}{n_0b_0} \) \right| \\
&\ll z^{-\del} \ll N_j^{-\del \mu (1/2-\eps)(1-\eps)/2}\log N_j
\end{align*}
and
$$
|f_1(a_1m+b_1)-\a_1| = |f_1(b_1 n_{1,j} (qs+r))-\a_1| \ll
N_j^{-\del\xi'}+N_j^{-\mu \del}.
$$
Hence
$$
|f_1(a_1m+b_1)-f_0(a_0m+b_0)-\zeta_1| \ll m^{-\frac{\del\xi'}{1+\mu}}
+ m^{-\frac{\del \mu (1/2-\eps)(1-\eps)}{2(1+\mu)}} \log m.
$$
Taking $\mu = \frac{2\xi'}{(1/2-\eps)(1-\eps)}$,
$\xi'$ close enough to $\lam$ and $\eps$ small enough
completes the proof.  Finally, if we assume the Elliott-Halberstam
conjecture, \eqref{rem} holds with $y=X^{1-\eps}$ and we have, for any
$c< \frac{\del\lam}{1+2\lam}$, that the inequality
$$
|f_1(a_1m+b_1)-f_0(a_0m+b_0)-\zeta_1| \ll m^{-c}
$$
holds for infinitely many $m$.

%
\section{Dealing with polynomial arguments}\label{sec:poly}
%

Let $f(n)=\log(n/\phi(n))$ or $f(n)=\log(\sigma(n)/n)$.  We have
$f(p)=1/p+O(1/p^2)$, so (b) holds with $\delta=1$.
Let $0<\xi < \lam < 1-\Gamma'$.  We have

(a') $\ds \sum_{\substack{p\equiv 1\pmod{4}}} f(p) = \infty$,

(c') If $t_0$ is small enough, then
for any $0<t\le t_0$, there is a prime $p\equiv 1\pmod{4}$ so that
$t-t^{1+\lam} \le f(p) \le t$.

We follow the proof of Theorem \ref{Theorem2} (in the case $k\ge 2$).
Let $n_0$ be the product of primes $\equiv 1,3\pmod{8}$ and such that
$f(n_0)>|\zeta|+1$, put $\a_0=f(2n_0)$ and $\a_1=\zeta+\a_0$.
Armed with (a') and (c'), an analog of Lemma \ref{Lemma1} holds with $k=1$ and
$\mathcal{P}_1$ consisting only of primes $\equiv 1\pmod4$, and an
analog of Lemma \ref{Lemma2} holds with the additional restriction
that $n_{1,j}$ is the product of only primes $\equiv 1\pmod{4}$.
By our construction of $n_0$ and $n_{1,j}$, the system of congruences
\begin{align*}
m &\equiv 0 \pmod{2} \\
m^2+2 &\equiv 0 \pmod{n_0} \\
m^2+1 &\equiv 0 \pmod{n_{1,j}} \\
\end{align*}
has at least one solution $m\equiv h_j\pmod{N_j}$ with
$N_j=2n_0n_{1,j}$ and $0\le h_j < N_j$.  Apply
Theorem DHR with
$$
\mathcal{A} = \left\{ \frac{m^2+2}{2n_0} \cdot \frac{m^2+1}{n_{1,j}} :
m=h_j+N_js, 1\le s\le N_j^{\mu} \right\}
$$
$\mathcal{P}$ the set of primes $\le z$, $X=N_j^\mu$, $\kappa=2$
(since $m^2+1$ and $m^2+2$ are irreducible
and coprime, $\omega(p)=2$ on average),
$y=X^{1-2\eps}$, $z=y^{\frac{1-\eps}{\b_2}}$.
There is an $m\le N_j^{1+\mu}+N_j$ so that the above system of
congruences holds, and all prime factors of
$\frac{m^2+2}{2n_0} \cdot \frac{m^2+1}{n_{1,j}}$ are $>z$.
The rest of the argument is the same as in the proof of Theorem
\ref{Theorem2}.

\bigskip

For the general problem of simultaneously approximating
$f_i(g_i(n))$ for $1\le i\le k$, each $g_i$ needs to satisfy
a version of (a') and (c')
where the primes are restricted to those for which $g_i(n)\equiv
0\pmod{p}$ has a solution.
Also, each $g_i$ should have an irreducible factor $g_i^*$ not dividing any
other $g_j$.  This way, the quantities $f_i(g_i^*(n))$ will be
sufficiently independent to allow the method to work.
Analogous to
Theorem 1, with appropriate restrictions on $\a_i$, the
system
$$
|f_i(g_i(n))-\a_i| < n^{-c} \qquad (1\le i\le k)
$$
will have infinitely many solutions for some $c>0$.  Here $c$ will
depend on $\delta$, $A$, $k$ and the number of irreducible factors of
each $g_i$ and each $(g_i,g_j)$, $i\ne j$.
Similarly, for any
$\zeta_1,\ldots,\zeta_k$ the system
$$
|f_i(g_i(n))-f_{i-1}(g_{i-1}(n))-\zeta_i| < n^{-c} \qquad (1\le i\le k)
$$
will have infinitely many solutions for some $c>0$.

%

\vspace{3mm}

\noindent{\small Department of Mathematics, Koc University,
Rumelifeneri Yolu, 34450, Sariyer, Istanbul, TURKEY.\\
e-mail: ealkan@ku.edu.tr}

\bigskip

\noindent{\small
Department of Mathematics,
University of Illinois at Urbana-Champaign,
1409 W. Green Street,
Urbana, IL, 61801, USA.\\
e-mail: ford@math.uiuc.edu\\
e-mail: zaharesc@math.uiuc.edu}


\begin{thebibliography}{99}
%


\bibitem{AHZ} E. Alkan, G. Harman and A. Zaharescu, {\it Diophantine
  approximation with mild divisibility constraints}, J. Number Theory
  {\bf 118} (2006), 1--14.

\bibitem{Ba} R. C. Baker, Diophantine Inequalities, London
  Math. Soc. Monogr. (N.S.), vol. 1, Oxford Univ. Press, New York, 1986.

\bibitem{BHP} R. C. Baker, G. Harman and J. Pintz,
{\it The difference between consecutive primes. II.},
Proc. London Math. Soc. (3) {\bf 83} (2001), no. 3, 532--562.


\bibitem{Da} H. Davenport, {\it Multiplicative number theory}, 3rd ed.,
Graduate Texts in Mathematics vol. 74, Springer-Verlag, New York, 2000.



\bibitem{DHR1} H.Diamond, H. Halberstam and H.~-E. Richert, {\it
Combinatorial sieves of dimension exceeding one}, J. Number Theory
{\bf 28} (1988), 306--346.

\bibitem{DHR2}  H.Diamond, H. Halberstam and H.~-E. Richert,
{\it Combinatorial sieves of dimension exceeding one II},
In {\bf Analytic Number Theory, Proceedings of a Conference in Honor of
Heini Halberstam}, 265--308 (B.C. Berndt, H. G. Diamond and
A. J. Hildebrand, eds.).  Birkh\"auser (1996).

\bibitem{E} P. Erd\H os, {\it Some remarks on Euler's $\phi$ function}, Acta
Arith. {\bf 4} (1958), 10--19.

\bibitem{ES} P. Erd\H os and A. Schinzel, {\it Distributions of the
 values of some arithmetical functions}, Acta Arith. {\bf 6},
(1960/61), 473--485.

\bibitem{EW} P. Erd\H os and A. Wintner, {\it Additive arithmetical
  functions and statistical independence}, Amer. J. Math. {\bf 61}
  (1939), 713--721.

\bibitem{Fr} J. B. Friedlander, {\it Fractional parts of sequences},
Th\'eorie des nombres (Quebec, PQ, 1987), 220--226, de Gruyter,
Berline, 1989.

\bibitem{GHP} S. W. Graham, J. J. Holt and C. Pomerance,
{\it On the solutions to $\phi(n)=\phi(n+k)$},
Number theory in progress, Vol. 2 (Zakopane-Ko\'scielisko, 1997), 867--882,
de Gruyter, Berlin, 1999.

\bibitem{HR}
H. Halberstam, H.-E. Richert, {\it Sieve methods}, LMS Monographs No. 4,
Academic Press, London (1974).

\bibitem{HKL}
G. Harman, A. Kumchev, and P. A. Lewis,
{\it The distribution of prime ideals of imaginary quadratic fields},
Trans. Amer. Math. Soc. {\bf 356} (2004), no. 2, 599--620.

\bibitem{LS}
F. Luca and I. E. Shparlinski, {\it Approximating positive reals by
  ratios of kernels of consecutive integers},
 Diophantine analysis and related fields 2006,  141--149,
Sem. Math. Sci., 35, Keio Univ., Yokohama, 2006.








\bibitem{Sc} A. Schinzel, {\it On functions $\phi(n)$ and $\sigma(n)$},
Bull. Acad. Pol. Sci. Cl. III {\bf 3} (1955), 415--419.

\bibitem{Scho} I. Schoenberg, {\it \"Uber die asymptotische Verteilung
  reeller Zahlen mod 1}, Mat. Z. {\bf 28} (1928), 171--199.



\bibitem{W} D. Wolke, {\it Eine Bemerkung \"uber die Werte der
Funktion $\sigma (n)$},
Monatsh. Math. {\bf 83} (1977), no. 2, 163--166. (German. English summary)

\end{thebibliography}
\end{document}